\documentclass{amsart}
\usepackage{amssymb, amsmath, amsthm, graphics, comment, xspace, enumerate}
\newtheorem{thm}{Theorem}[section]
\newtheorem{cor}[thm]{Corollary}
\newtheorem{lem}[thm]{Lemma}
\newtheorem{prop}[thm]{Proposition}
\theoremstyle{definition}
\newtheorem{defn}[thm]{Definition}
\newtheorem{example}[thm]{Example}
\theoremstyle{remark}
\newtheorem{rem}[thm]{Remark}
\numberwithin{equation}{section}

\begin{document}
\title[${\mathcal F}$-Hypercyclic operators on Fr\' echet spaces]{${\mathcal F}$-Hypercyclic operators on Fr\' echet spaces}

\author{Marko Kosti\' c}
\address{Faculty of Technical Sciences,
University of Novi Sad,
Trg D. Obradovi\' ca 6, 21125 Novi Sad, Serbia}
\email{marco.s@verat.net}

{\renewcommand{\thefootnote}{} \footnote{
\\ \text{  }  \ \    {\it Key words and phrases.} ${\mathcal F}$-hypercyclicity, $(m_{n})$-hypercyclicity, $q$-frequent hypercyclicity.
\\  \text{  }  \ \ The author is partially supported by grant 174024 of Ministry
of Science and Technological Development, Republic of Serbia.}}

\begin{abstract}
In this paper, we investigate ${\mathcal F}$-hypercyclicity of linear, not necessarily continuous, operators on Fr\' echet spaces. The  notion of
lower $(m_{n})$-hypercyclicity seems to be new and not considered elsewhere even for linear continuous operators acting on Fr\' echet spaces. We pay special attention 
to the study of $q$-frequent hypercyclicity, where $q\geq 1$ is an arbitrary real number. We present several new concepts and results for lower and upper densities in a separate section, providing also a great number of illustrative examples and open problems.
\end{abstract}

\maketitle

\section{Introduction and preliminaries}

Let $X$ be a Fr\' echet space.
A linear operator $T$ on $X$ is said to be hypercyclic
iff there exists an element $x\in D_{\infty}(T)$ whose orbit $\{T^nx:n\in\mathbb{N}_0\}$ is dense in $X$;
$T$ is said to be topologically transitive, resp. topologically mixing,
iff for every pair of open non-empty subsets $U,\,V$ of $X$,
there exists $n\in\mathbb{N}$ such that $T^n(U)\cap V\neq\emptyset$,
resp. iff for every pair of open non-empty subsets $U,\,V$ of $X$, there exists $n_0\in\mathbb{N}$
such that, for every $n\in\mathbb{N}$ with $n\geq n_{0}$, we have $T^n(U)\cap V\neq\emptyset$. By a chaotic operator, we mean any linear hypercyclic operator $T$ for which 
the set of periodic points of $T,$ defined by $\{x\in D_{\infty}(T) : (\exists n_{0}\in {\mathbb N}) \,  T^{n_{0}}x=x\},$ is dense in $X$.
For more details about linear topological dynamics, we refer the reader to \cite{Bay} and \cite{Grosse}.

The notion of an $(m_{n})$-hypercyclic operator on a separable Fr\' echet space was introduced by Bayart and Matheron \cite{Bayart} in 2009 with a view to control the frequency of the individual orbits of a hypercyclic, non-weakly mixing operator. 
Gupta, Mundayadan \cite{gupta}-\cite{gupta-prim} and Heo, Kim-Kim \cite{heo} have been recently considered the special case of 
$(m_{n})$-hypercyclicity, the so-called $q$-frequent hypercyclicity, where the sequence $(m_{n})$ is given by $m_{n}:=n^{q}$ ($q\in {\mathbb N}$). In the case that $q=1,$ the $q$-frequent hypercyclicity
is also known as frequent hypercyclicity and, without any doubt, that is the best explored concept of above-mentioned. 

Suppose that ${\mathcal F}$ is a non-empty collection of certain subsets of ${\mathbb N},$ i.e., ${\mathcal F}\in P(P({\mathbb N}))$ and ${\mathcal F}\neq \emptyset.$ The main aim of this paper is to investigate ${\mathcal F}$-hypercyclic properties of linear operators on Fr\' echet spaces. The general case is a very non-trivial and, because of that, we assume that the family ${\mathcal F}$ takes some special forms. In Section 2, we introduce and further analyze the notions of lower and upper $(m_{n})$-densities,
$l;(m_{n})$-Banach densities and
$u;(m_{n})$-Banach densities (the results established in this section can be viewed of some independent interest). Strictly speaking, in this paper, we analyze the case in which ${\mathcal F}$ is a collection of all non-empty subsets of ${\mathbb N}$ having the positive lower or upper $(m_{n})$-density,
$l;(m_{n})$-Banach density or
$u;(m_{n})$-Banach density. The notion of lower $(m_{n})$-hypercyclicity is introduced in this direction.
We focus our attention to the case that  $m_{n}=n^{q}$ for some real number $q\geq 1,$ and slightly improve several structural results from \cite{gupta}-\cite{heo} in this context
(see Section 3). We formulate l-$(m_{n})$-frequent universality criterion for unbounded linear operators and upper frequent hypercyclicity criterion for unbounded linear operators obeying the methods from the theory of $C$-regularized semigroups (see deLaubenfels, Emamirad, Grosse-Erdmann \cite{cycch} for pioneering results in this direction).

The author would like to thank Prof. M. Murillo-Arcila (Universitat Jaume I, Castell\' on, Spain) and Y. Puig de Dios (University of California, Riverside, USA) for many stimulating discussions during the genesis of paper.

We use the standard notation throughout the paper. For any $s\in {\mathbb R},$ we define $\lfloor s \rfloor :=\sup \{
l\in {\mathbb Z} : s\geq l \}$ and $\lceil s \rceil :=\inf \{ l\in
{\mathbb Z} : s\leq l \}.$ Unless stated otherwise, in this paper we assume that $X$ is a Fr\' echet space (real or complex) and
the topology of $X$ is induced by the fundamental system
$(p_{n})_{n\in {\mathbb N}}$ of increasing seminorms. If $Y$ is also a Fr\' echet space, over the same field of scalars as $X,$ then by $L(X,Y)$ we denote the space consisting of all continuous linear mappings from $X$ into $Y.$ The
translation invariant metric $d: X\times X \rightarrow [0,\infty),$ defined by
$$
d(x,y):=\sum
\limits_{n=1}^{\infty}\frac{1}{2^{n}}\frac{p_{n}(x-y)}{1+p_{n}(x-y)},\
x,\ y\in X,
$$
satisfies, among many other properties, the following:
$d(x+u,y+v)\leq d(x,y)+d(u,v)$ and $d(cx,cy)\leq (|c|+1)d(x,y),\
c\in {\mathbb K},\ x,\ y,\ u,\ v\in X.$ Set $L(x,\epsilon):=\{y\in X : d(x,y)<\epsilon\}.$ For a closed linear operator $T$ on $X,$ we denote by $D(T),$ $R(T),$ $N(T)$ and $\rho(T)$ its domain, range, kernel and resolvent set, respectively. Set $D_{\infty}(T):=\bigcap_{k\in {\mathbb N}}D(T^{k}).$ If $T^{k}$ is closed for any $k\in {\mathbb N},$ then the space
$D(T^{k}),$ equipped with the following family of seminorms $p_{k,n}(x):=p_{n}(x)+p_{n}(Tx)+\cdot \cdot \cdot +p_{n}(T^{k}x),$ $x\in D(T^{k}),$ is a Fr\'echet one. The projective limit of spaces $[D(T^{k})]$ as $k\rightarrow \infty$ is a Fr\'echet space, which will be denoted by $[D_{\infty}(T)].$ 
We will always assume henceforth that $C\in L(X)$ and $C$ is injective. Put $p_{C}(x):=p(C^{-1}x),$ $p\in \circledast,$ $x\in R(C).$ Then
$p_{C}(\cdot)$ is a seminorm on $R(C)$ and the calibration
$(p_{C})_{p\in \circledast}$ induces a Fr\'echet topology on
$R(C);$ we denote this space by $[R(C)]_{\circledast}.$

Let us recall that a series $\sum_{n=1}^{\infty}x_{n}$ in $X$ is called unconditionally convergent iff for every permutation $\sigma$ of ${\mathbb N}$,
the series $\sum_{n=1}^{\infty}x_{\sigma (n)}$  is convergent (see \cite{boni} and references cited therein for further information on the subject).
A collection of series
$\sum_{n=1}^{\infty}x_{n,k},$ $k\in J$ is called
unconditionally convergent, uniformly in $k\in J$ iff for any $\epsilon>0$ there exists
$N\in {\mathbb N}$ such that for any finite set $F\subseteq [N,\infty) \cap {\mathbb N}$ and for every $k\in J$ we have $\sum_{n\in F}x_{n,k} \in L(0,\epsilon).$ 

Let ${\mathcal F}$ be a non-empty collection of certain subsets of ${\mathbb N},$ i.e., ${\mathcal F}\in P(P({\mathbb N}))$ and ${\mathcal F}\neq \emptyset.$ Observe that, in contrast to \cite{puid-dios-erg}, we do not require here that $|A|=\infty$ for all $A\in{\mathcal F}$ as well as that ${\mathcal F}$ satisfies the following property:
\begin{itemize}
\item[(I)] $B\in{\mathcal F}$ whenever there exists $A\in{\mathcal F}$ such that $A\subseteq B.$
\end{itemize}
If ${\mathcal F}$ satisfies (I), then it is said that 
${\mathcal F}$ is a Furstenberg family; a proper Furstenberg family ${\mathcal F}$ is any Furstenberg family satisfying that $\emptyset \notin {\mathcal F}.$ For the sequel, we also need the notion of an upper Furstenberg family; that is any proper Furstenberg family ${\mathcal F}$ satisfying the following two conditions:
\begin{itemize}
\item[(II)] There exist a set $D$ and a countable set $M$ such that ${\mathcal F}=\bigcup_{\delta \in D} \bigcap_{\nu \in M}{\mathcal F}_{\delta,\nu},$ where for each $\delta \in D$ and $\nu \in M$ the following holds: If $A\in {\mathcal F}_{\delta,\nu},$ then there exists a
finite subset $F\subseteq {\mathbb N}$ such that the implication $A\cap F \subseteq B \Rightarrow B\in {\mathcal F}_{\delta,\nu}$ holds true.
\item[(III)] If $A\in {\mathcal F},$ then there exists $\delta \in D$ such that, for every $n\in {\mathbb N},$ we have $A-n\equiv \{k-n: k\in A,\ k>n\}\in {\mathcal F}_{\delta},$ where ${\mathcal F}_{\delta}\equiv \bigcap_{\nu \in M}{\mathcal F}_{\delta,\nu}.$
\end{itemize} 

We would like to propose the following definition:

\begin{defn}\label{4-skins-MLO-okay}
Let $(T_{n})_{n\in {\mathbb N}}$ be a sequence of linear operators acting between the spaces
$X$ and $Y,$ let $T$ be a linear operator on $X$, and let $x\in X$. Suppose that ${\mathcal F}\in P(P({\mathbb N}))$ and ${\mathcal F}\neq \emptyset.$ Then we say
that:
\begin{itemize}
\item[(i)] $x$ is an ${\mathcal F}$-hypercyclic element of the sequence
$(T_{n})_{n\in {\mathbb N}}$ iff
$x\in \bigcap_{n\in {\mathbb N}} D(T_{n})$ and for each open non-empty subset $V$ of $Y$ we have that 
$$
S(x,V):=\bigl\{ n\in {\mathbb N} : T_{n}x \in V \bigr\}\in {\mathcal F} ;
$$ 
$(T_{n})_{n\in {\mathbb N}}$ is said to be ${\mathcal F}$-hypercyclic iff there exists an ${\mathcal F}$-hypercyclic element of
$(T_{n})_{n\in {\mathbb N}}$;
\item[(ii)] $T$ is ${\mathcal F}$-hypercyclic iff the sequence
$(T^{n})_{n\in {\mathbb N}}$ is ${\mathcal F}$-hypercyclic; $x$ is said to be
an ${\mathcal F}$-hypercyclic element of $T$ iff $x$ is an ${\mathcal F}$-hypercyclic element of the sequence
$(T^{n})_{n\in {\mathbb N}}.$
\end{itemize}
\end{defn}

In the following definition, we extend the notion introduced by F. Bayart and \'E. Matheron \cite[Definition 1.1]{Bayart}, given for continuous linear operators on separable Fr\' echet spaces:

\begin{defn}\label{prost}
Let $(T_{j})_{j\in {\mathbb N}}$ be a sequence of linear operators between the spaces
$X$ and $Y,$ let $T$ be a linear operator on $X$, and let $x\in X$. Suppose that $(m_{n})$ is an increasing sequence in $[1,\infty).$ Then we say
that:
\begin{itemize}
\item[(i)] $x$ is an $(m_{n})$-hypercyclic element of the sequence
$(T_{j})_{j\in {\mathbb N}}$ iff
$x\in  D(T_{j}),$ $j\in {\mathbb N}$ and, for every
$n\in {\mathbb N}_{0}$ and for every open non-empty subset $V$ of $Y,$ we have that 
there exist a strictly increasing sequence $(l_{n})$ of positive integers and a finite constant $L>0$ such that
$l_{n}\leq Lm_{n},$ $n\in {\mathbb N}$ and $T_{l_{n}}x\in V;$
$(T_{j})_{j\in {\mathbb N}}$ is said to be $(m_{n})$-hypercyclic iff there exists an $(m_{n})$-hypercyclic element of
$(T_{j})_{j\in {\mathbb N}}$;
\item[(ii)] $x$ is said to be
an $(m_{n})$-hypercyclic element of $T$ iff $x$ is an $(m_{n})$-hypercyclic element of the sequence
$(T^{j})_{j\in {\mathbb N}}$; $T$ is $(m_{n})$-hypercyclic iff the sequence
$(T^{j})_{j\in {\mathbb N}}$ is $(m_{n})$-hypercyclic.
\end{itemize}
\end{defn}

\section{Lower and upper densities}\label{densities}

It is well known that the lower density of a non-empty subset $A\subseteq {\mathbb N},$ denoted by $\underline{d}(A),$ is defined through:
$$
\underline{d}(A):=\liminf_{n\rightarrow \infty}\frac{|A \cap [1,n]|}{n},
$$
as well as that the upper density of $A,$ denoted by $\overline{d}(A),$ is defined through:
$$
\overline{d}(A):=\limsup_{n\rightarrow \infty}\frac{|A \cap [1,n]|}{n}.
$$
Further on, the lower Banach density of $A,$ denoted by $\underline{Bd}(A),$ is defined through:
$$
\underline{Bd}(A):=\lim_{s\rightarrow +\infty}\liminf_{n\rightarrow \infty}\frac{|A \cap [n+1,n+s]|}{s},
$$
while the (upper) Banach density of $A,$ denoted by $\overline{Bd}(A),$ is defined through:
$$
\overline{Bd}(A):=\lim_{s\rightarrow +\infty}\limsup_{n\rightarrow \infty}\frac{|A \cap [n+1,n+s]|}{s}.
$$
It is well known that 
\begin{align}\label{jebu}
0\leq \underline{Bd}(A) \leq \underline{d}(A) \leq \overline{d}(A) \leq \overline{Bd}(A)\leq 1,
\end{align}
as well as that
${\mathcal F}:=\{ A\subseteq {\mathbb N} : \overline{d}(A)>0 \}$ (${\mathcal F}:=\{ A\subseteq {\mathbb N} : \overline{Bd}(A)>0 \}$) is an upper Furstenberg family; see e.g. \cite{grekos-prim}.
For more details about the above densities, we refer the reader to \cite{grekos}, \cite{hindu}, \cite{dens} and references cited therein. 

In this section, we introduce the following notions of lower and upper densities for a subset $A\subseteq {\mathbb N}:$ 

\begin{defn}\label{prckojed}
Let $q\in [1,\infty),$ and let $(m_{n})$ be an increasing sequence in $[1,\infty).$ Then we define:
\begin{itemize}
\item[(i)] The lower $q$-density of $A,$ denoted by $\underline{d}_{q}(A),$ is defined through:
$$
\underline{d}_{q}(A):=\liminf_{n\rightarrow \infty}\frac{|A \cap [1,n^{q}]|}{n}.
$$
\item[(ii)] The upper $q$-density of $A,$ denoted by $\overline{d}_{q}(A),$ is defined through:
$$
\overline{d}_{q}(A):=\limsup_{n\rightarrow \infty}\frac{|A \cap [1,n^{q}]|}{n}.
$$
\item[(iii)] The lower $(m_{n})$-density of $A,$ denoted by $\underline{d}_{m_{n}}(A),$ is defined through:
$$
\underline{d}_{{m_{n}}} (A):=\liminf_{n\rightarrow \infty}\frac{|A \cap [1,m_{n}]|}{n}.
$$
\item[(iv)] The upper $(m_{n})$-density of $A,$ denoted by $\overline{d}_{{m_{n}}}(A),$ is defined through:
$$
\overline{d}_{{m_{n}}}(A):=\limsup_{n\rightarrow \infty}\frac{|A \cap [1,m_{n}]|}{n}.
$$
\end{itemize}
\end{defn}

The notion in (i)-(ii) has been considered recently by Gupta and Mundayadan \cite{gupta} in the case that $q\in {\mathbb N},$ while this notion seems to be new provided that $q \in [1,\infty) \setminus {\mathbb N}.$ To the best knowledge of the author, the notion in (iii)-(iv)  has not been considered elsewhere.

The proof of following simple result is very similar to that of \cite[Proposition 3.2]{gupta}, given in the case that $q\in{\mathbb N},$ and therefore omitted:

\begin{prop}\label{perhane}
Suppose that  $q\geq 1,$ $A=\{ n_{1},\ n_{2}, \cdot \cdot \cdot,\  n_{k},\cdot \cdot \cdot \},$ where $(n_{k})$ is a strictly increasing sequence of positive integers. Then $\underline{d}_{q}(A)=\liminf_{k\rightarrow \infty}\frac{k}{n_{k}^{1/q}}$ and
$\underline{d}_{q}(A)>0$ iff there exists a finite constant $L>0$ such that $n_{k}\leq Lk^{q},$ $k\in {\mathbb N}.$ 
\end{prop}

For the sequel, it would be worthwhile to observe that for each $q\geq 1$ we have:
\begin{align}\label{nejes}
[0,1]\ni \underline{d}(A) \leq \underline{d}_{q}(A) \in [0,\infty] \ \ \mbox{ and }\ \ [0,1]\ni \overline{d}(A) \leq \overline{d}_{q}(A) \in [0,\infty].
\end{align}

Further on, it is very simple to prove that 
${\mathcal F}:=\{ A\subseteq {\mathbb N} : \overline{d}_{{m_{n}}}(A)>0 \}$ is a (proper) Furstenberg family iff ${\mathbb N} \in {\mathcal F},$ i.e., if $\limsup_{n\rightarrow \infty}\frac{m_{n}}{n}>0.$ 
In this case, we have
${\mathcal F}=\bigcup_{\delta>0}\bigcap_{n\in {\mathbb N}}{\mathcal F}_{\delta,n},$ where ${\mathbb N} \supseteq A\in {\mathcal F}_{\delta,n}$ iff there exists $N\geq n$ such that $|A \cap [1,m_{N}]|/N >\delta.$ Hence, the condition (II) holds.
The condition (III) also holds because for each $k\in {\mathbb N}$ we have 
\begin{align*}
\frac{\bigl|A \cap [1,m_{N}]\bigr|-k}{N} \leq  \frac{\bigl|A_{k} \cap [1,m_{N}]\bigr|}{N} \leq \frac{\bigl|A \cap [1,m_{N}]\bigr|+k}{N}
\end{align*}
so that $
\overline{d}_{{m_{n}}}(A)=
\overline{d}_{{m_{n}}}(A_{k})$ ($
\underline{d}_{{m_{n}}}(A)=
\underline{d}_{{m_{n}}}(A_{k})$) and ${\mathcal F}$ is an upper Furstenberg family. Therefore, we have proved the following:

\begin{prop}\label{tasta}
Let $(m_{n})$ be an increasing sequence in $[1,\infty),$ and let $A\subseteq {\mathbb N}.$ Then ${\mathcal F}:=\{ A\subseteq {\mathbb N} : \overline{d}_{{m_{n}}}(A)>0 \}$ is a (proper, upper) Furstenberg family iff $\limsup_{n\rightarrow \infty}\frac{m_{n}}{n}>0.$ 
\end{prop}

We also introduce the following notion:

\begin{defn}\label{guptinjo}
Suppose $q\in [1,\infty),$ $(m_{n})$ is an increasing sequence in $[1,\infty)$ and $A\subseteq {\mathbb N}.$ Then we define:
\begin{itemize}
\item[(i)] The lower $l;q$-Banach density of $A,$ denoted shortly by $\underline{Bd}_{l;q}(A),$
as follows
$$
\underline{Bd}_{l;q}(A):=\liminf_{s\rightarrow +\infty}\liminf_{n\rightarrow \infty}\frac{|A \cap [n+1,n+s^{q}]|}{s}.
$$ 
\item[(ii)] The lower $u;q$-Banach density of $A,$ denoted shortly by $\underline{Bd}_{u;q}(A),$
as follows
$$
\underline{Bd}_{u;q}(A):=\limsup_{s\rightarrow +\infty}\liminf_{n\rightarrow \infty}\frac{|A \cap [n+1,n+s^{q}]|}{s}.
$$ 
\item[(iii)] The $l;q$-Banach density of $A,$ denoted shortly by $\overline{Bd}_{l;q}(A),$
as follows
$$
\overline{Bd}_{l;q}(A):=\liminf_{s\rightarrow +\infty}\limsup_{n\rightarrow \infty}\frac{|A \cap [n+1,n+s^{q}]|}{s}.
$$ 
\item[(iv)] The $u;q$-Banach density of $A,$ denoted shortly by $\overline{Bd}_{u;q}(A),$
as follows
$$
\overline{Bd}_{u;q}(A):=\limsup_{s\rightarrow +\infty}\limsup_{n\rightarrow \infty}\frac{|A \cap [n+1,n+s^{q}]|}{s}.
$$ 
\item[(v)] The lower $l;(m_{n})$-Banach density of $A,$ denoted shortly by $\underline{Bd}_{l;{m_{n}}}(A),$
as follows
$$
\underline{Bd}_{l;{m_{n}}}(A):=\liminf_{s\rightarrow +\infty}\liminf_{n\rightarrow \infty}\frac{|A \cap [n+1,n+m_{s}]|}{s}.
$$ 
\item[(vi)] The lower $u;(m_{n})$-Banach density of $A,$ denoted shortly by $\underline{Bd}_{u;{m_{n}}}(A),$
as follows
$$
\underline{Bd}_{u;{m_{n}}}(A):=\limsup_{s\rightarrow +\infty}\liminf_{n\rightarrow \infty}\frac{|A \cap [n+1,n+m_{s}]|}{s}.
$$ 
\item[(vii)] The (upper) $l;(m_{n})$-Banach density of $A,$ denoted shortly by $\overline{Bd}_{l;{m_{n}}}(A),$
as follows
$$
\overline{Bd}_{l;{m_{n}}}(A):=\liminf_{s\rightarrow +\infty}\limsup_{n\rightarrow \infty}\frac{|A \cap [n+1,n+m_{s}]|}{s}.
$$ 
\item[(viii)] The (upper) $u;(m_{n})$-Banach density of $A,$ denoted shortly by $\overline{Bd}_{u;{m_{n}}}(A),$
as follows
$$
\overline{Bd}_{u;{m_{n}}}(A):=\limsup_{s\rightarrow +\infty}\limsup_{n\rightarrow \infty}\frac{|A \cap [n+1,n+m_{s}]|}{s}.
$$ 
\end{itemize}
\end{defn}

Any of the densities introduced in the parts (i)-(viii) is translation invariant in the sense that for each $k\in {\mathbb N}$ the sets $A+k:=\{a+k : a\in A\}$ and $A-k:=\{a-k : a\in A,\ a>k\}$ has the same density as $A;$ in contrast to this, it is not true that
the density of $k A:=\{ka : a\in A\}$ equals $k^{-1}$ times the density of $A,$ in general ($k\in {\mathbb N}$).
As the next obvious lemma shows, it is sufficient to introduce the notion of lower and upper $(m_{n})$-density of $A$ as well as the notion of lower and upper $u;(m_{n})$-Banach density of $A$ only
in the case that $(m_{n})$ is a sequence of positive integers:

\begin{lem}\label{zelje-zelje}
Let $(m_{n})$ be an increasing sequence  in $[1,\infty),$ and let $A\subseteq {\mathbb N}.$ Then the following holds:
\begin{itemize}
\item[(i)]
$\underline{d}_{{m_{n}}}(A)=\underline{d}_{\lceil m_{n} \rceil}(A)=\underline{d}_{\lfloor m_{n} \rfloor}(A)$ and $\overline{d}_{{m_{n}}}(A)=\overline{d}_{\lceil m_{n} \rceil}(A)=\overline{d}_{\lfloor m_{n} \rfloor}(A).$
\item[(ii)] $\overline{Bd}_{u;{m_{n}}}(A)=\overline{Bd}_{u;{\lfloor m_{n} \rfloor}}(A)=
\overline{Bd}_{u;{\lceil m_{n} \rceil}}(A)$ and\\ $\underline{Bd}_{u;{m_{n}}}(A)=\underline{Bd}_{u;{\lfloor m_{n} \rfloor}}(A)=
\underline{Bd}_{u;{\lceil m_{n} \rceil}}(A).$
\end{itemize}
\end{lem} 

\begin{proof}
Follows simply from the next estimates ($n\in {\mathbb N}$):
$$
\frac{\bigl|A \cap [1,\lfloor m_{n} \rfloor ]\bigr|}{n}\leq \frac{\bigl|A \cap [1,m_{n}]\bigr|}{n} \leq \frac{\bigl|A \cap [1,\lceil m_{n} \rceil ]\bigr|}{n}\leq \frac{1+\bigl|A \cap [1,\lfloor m_{n} \rfloor ]\bigr|}{n}.
$$
and
\begin{align*}
\frac{\bigl|A \cap [n+1,n+\lfloor m_{s}\rfloor]\bigr|}{s} &\leq \frac{\bigl|A \cap [n+1,n+m_{s}]\bigr|}{s} 
\\ & \leq \frac{\bigl|A \cap [n+1,n+\lceil m_{s}\rceil]\bigr|}{s} \leq \frac{1+\bigl|A \cap [n+1,n+\lfloor m_{s} \rfloor]\bigr|}{s}.
\end{align*}
\end{proof}

Further on, the quantities $
\underline{Bd}_{l;{m_{n}}}(A)$ and $
\underline{Bd}_{u;{m_{n}}}(A)
$ are very unexciting if $\lim_{s\rightarrow +\infty}\frac{m_{s}}{s}=+\infty$ or $\lim_{s\rightarrow +\infty}\frac{m_{s}}{s}=0,$ when we have the following:

\begin{prop}\label{mile}
Suppose that $A\subseteq {\mathbb N},$  $(m_{n})$ is an increasing sequence in $[1,\infty)$ and $\lim_{s\rightarrow +\infty}\frac{m_{s}}{s}=+\infty$ or $\lim_{s\rightarrow +\infty}\frac{m_{s}}{s}=0.$ Then 
$$
\underline{Bd}_{l;{m_{n}}}(A)=
\underline{Bd}_{u;{m_{n}}}(A)\in \{0,+\infty\}.
$$
\end{prop}

\begin{proof}
The proof is obvious if the set $A$ is finite, when we have $
\underline{Bd}_{l;{m_{n}}}(A)=
\underline{Bd}_{u;{m_{n}}}(A)=0.$ If the set $A$ is infinite, then there exists a strictly increasing sequence $(d_{n})$ of positive integers such that $A=\{d_{n} : n\in {\mathbb N}\}.$ Set $e_{n}:=d_{n+1}-d_{n}$ ($n\in {\mathbb N}$). 
Clearly, if  $\lim_{s\rightarrow +\infty}\frac{m_{s}}{s}=0,$ then for each $s,\ n\in {\mathbb N}$ we have $|A \cap [n+1,n+m_{s}]|\leq m_{s}+1$ and therefore $
\underline{Bd}_{l;{m_{n}}}(A)=
\underline{Bd}_{u;{m_{n}}}(A)=0.$ Suppose now that $\lim_{s\rightarrow +\infty}\frac{m_{s}}{s}=+\infty .$
Then there exist two possibilities:
\begin{itemize}
\item[1.] The sequence $(e_{n})$ is bounded from above by a finite constant $M\geq 1.$ Then $d_{n}\leq d_{1}+(n-1)M,$ $n\in {\mathbb N}$ and any subinterval in $[d_{1},\infty)$ of length $2M$ contains at least one element of $A$, so that $|A \cap [n+1,n+m_{s}]| \geq  (m_{s}-2)/2M,$ $s,\ n\in {\mathbb N}$ and  $
\underline{Bd}_{l;{m_{n}}}(A)=
\underline{Bd}_{u;{m_{n}}}(A)=+\infty$ since $\lim_{s\rightarrow +\infty}\frac{m_{s}}{s}=+\infty.$
\item[2.] There exists a subsequence $(e_{n_{k}})$ of $(e_{n})$ tending to $+\infty.$ In this case, we have $
\underline{Bd}_{l;{m_{n}}}(A)=
\underline{Bd}_{u;{m_{n}}}(A)=0.$ To verify this, it is enough to show that for each fixed number $s\in {\mathbb N}$ we have 
$$
\liminf_{n\rightarrow +\infty}\frac{\bigl|A \cap [n+1,n+m_{s}]\bigr|}{s}\leq \frac{1}{s}.
$$
This follows from the fact that there exists a number $k_{0}$ such that $m_{s}-1 <e_{n_{k}}$ for all $k\geq k_{0},$ which clearly implies
$$
\lim_{k\rightarrow +\infty}\frac{\bigl|A \cap [d_{n_{k}},d_{n_{k}}+ m_{s}-1]\bigr|}{s}= \frac{1}{s}.
$$
\end{itemize}  
\end{proof}

In the second case, we have not used any of conditions $\lim_{s\rightarrow +\infty}\frac{m_{s}}{s}=+\infty$ or $\lim_{s\rightarrow +\infty}\frac{m_{s}}{s}=0$ (see also \cite{bay11}):

\begin{prop}\label{mile-label}
Suppose that there exists a strictly increasing sequence $(d_{n})$ of positive integers such that $A=\{d_{n} : n\in {\mathbb N}\}.$ Set $e_{n}:=d_{n+1}-d_{n}$ ($n\in {\mathbb N}$).  
If the sequence $(e_{n})$ is unbounded, then $
\underline{Bd}_{l;{m_{n}}}(A)=
\underline{Bd}_{u;{m_{n}}}(A)=0.$ 
\end{prop}

\begin{rem}\label{qad}
With the requirements of Proposition \ref{mile-label} being satisfied, we have that $\underline{d}_{{m_{n}}} (A)$ can be strictly positive or equal to zero. For example, if $A:=\{n^{2} : n\in {\mathbb N}\}$ and $m_{n}:=n^{2}$ ($n\in {\mathbb N}$), then  $\underline{d}_{{m_{n}}} (A)=1,$ while in the case that $A:=\{n^{2} : n\in {\mathbb N}\}$ and $m_{n}:=n$ ($n\in {\mathbb N}$), we have $\underline{d}_{{m_{n}}} (A)=0.$
\end{rem}

Let $A$ and $(e_{n})$ be defined as above, and let $(e_{n})$ be bounded. In the case that $\lim_{s\rightarrow +\infty}\frac{m_{s}}{s}=+\infty ,$ we have $
\underline{Bd}_{l;{m_{n}}}(A)=
\underline{Bd}_{u;{m_{n}}}(A)=+\infty$ but also $
\underline{d}_{{m_{n}}} (A)=+\infty$ since
the interval $[1,m_{n}]$ contains at least $\lceil cm_{n}\rceil$ elements of the set $A$ for some constant $c\in (0,1)$ and therefore 
$
\liminf_{n\rightarrow \infty}\frac{|A \cap [1,m_{n}]|}{n}\geq c\liminf_{n\rightarrow \infty}\frac{m_{n}}{n}=+\infty.$ Furthermore, if $\liminf_{s\rightarrow +\infty}\frac{m_{s}}{s}=0,$ then it is easily seen that  $
\underline{d}_{{m_{n}}} (A)=0.$
Taking into account the cases examined in the proof of Proposition \ref{mile}, we can clarify the following:

\begin{thm}\label{mile-prim}
Suppose that $A\subseteq {\mathbb N}$ and $(m_{n})$ is an increasing sequence in $[1,\infty).$ Then the following holds:
\begin{itemize}
\item[(i)]
If $\liminf_{s\rightarrow +\infty}\frac{m_{s}}{s}=0,$ then $
\underline{Bd}_{l;{m_{n}}}(A)= \underline{d}_{{m_{n}}} (A)=0.$
\item[(ii)] If $\lim_{s\rightarrow +\infty}\frac{m_{s}}{s}=0,$ then 
$
\underline{Bd}_{l;{m_{n}}}(A)=\underline{Bd}_{u;{m_{n}}}(A) = \underline{d}_{{m_{n}}} (A)=0.$ 
\item[(iii)] If $\lim_{s\rightarrow +\infty}\frac{m_{s}}{s}=+\infty,$ then $
\underline{Bd}_{l;{m_{n}}}(A)=\underline{Bd}_{u;{m_{n}}}(A) = \underline{d}_{{m_{n}}} (A)$  
and the common value of these terms equals $+\infty ,$ if there exists a strictly increasing sequence $(d_{n})_{n\in {\mathbb N}}$ of positive integers such that $A=\{d_{n} : n\in {\mathbb N}\}$ and the sequence $(e_{n}:=d_{n+1}-d_{n})_{n\in {\mathbb N}}$ is bounded, or $0,$ otherwise. 
\end{itemize}
\end{thm}

Since it is well known that $
\underline{Bd}_{l;q}(A)=
\underline{Bd}_{u;q}(A)=
\underline{Bd}(A) \in [0,1]$ for $q=1$ (see e.g. \cite{grekos}), the above theorem in combination with \eqref{jebu} immediately implies the following:

\begin{cor}\label{mile-duo}
Let $A\subseteq {\mathbb N}.$
\begin{itemize}
\item[(i)] Suppose that $q=1.$ Then $0\leq 
\underline{Bd}_{l;q}(A)=
\underline{Bd}_{u;q}(A)=\underline{Bd}(A) \leq \underline{d}(A) \leq 1.$
\item[(ii)] Suppose that $q> 1.$ Then $
\underline{Bd}_{l;q}(A)=
\underline{Bd}_{u;q}(A):=
\underline{Bd}_{q}(A) =\underline{d}_{q}(A)$ and the common value of these terms equals $+\infty ,$ if there exists a strictly increasing sequence $(d_{n})_{n\in {\mathbb N}}$ of positive integers such that $A=\{d_{n} : n\in {\mathbb N}\}$ and the sequence $(e_{n}:=d_{n+1}-d_{n})_{n\in {\mathbb N}}$ is bounded, or $0,$ otherwise. 
\end{itemize}
\end{cor}

For the sequel, let us note that $
\overline{Bd}_{l;q}(A)=
\overline{Bd}_{u;q}(A):=
\overline{Bd}_{q}(A)\in [0,1]$ provided that $q=1,$ when also the equality $\underline{Bd}_{q}(A)+\overline{Bd}_{q}({\mathbb N} \setminus A)=1$ takes the place (\cite{grekos}). This is no longer true in the case that $q>1:$ 

\begin{example}\label{gustish} 
\begin{itemize}
\item[(i)] Let $(s_{k}')$ be any increasing sequence of positive integers satisfying that $s_{k}'+(s_{k}')^{q}<s_{k+1}',$ $k\in {\mathbb N}$ as well as
$$
s_{k}'
\geq k \bigl(s_{k}'\bigr)^{\frac{q-1}{2q}} +2\bigl(s_{k}'\bigr)^{\frac{q-1}{2q(q+\epsilon -1)}},\quad k \in {\mathbb N},
$$
for some $\epsilon>0.$ Set $s_{k}:=\lfloor (s_{k}')^{\frac{q-1}{2q(q+\epsilon -1)}}\rfloor $ and $l_{k}:=\lfloor (s_{k}'/s_{k})^{q}  \rfloor$ for all $k\in {\mathbb N}.$ Let $A\subseteq {\mathbb N}$ have no elements outside the set $\bigcup_{k\in {\mathbb N}}[s_{k}'+1,s_{k}'+(s_{k}')^{q}],$ and let $s_{k}^{1-\epsilon}\leq |A \cap [s_{k}'+ls_{k}^{q},s_{k}'+(l+1)s_{k}^{q}]| \leq s_{k}^{1-\epsilon }+1$ for all $k\in {\mathbb N}$ and $0\leq l\leq l_{k}.$ Then it can be easily seen that, for every $n\in {\mathbb N},$ we have $|A\cap [n+1,n+s_{k}^{q}]| \leq 2(1+s_{k}^{1-\epsilon}),$ $n\in {\mathbb N}$ so that $
\overline{Bd}_{l;q}(A)=0.$ On the other hand, we have that $|A \cap [s_{k}'+1,s_{k}'+(s_{k}')^{q}]| \geq l_{k}s_{k}^{1-\epsilon} \geq ks_{k}'$; hence, $
\overline{Bd}_{u;q}(A)=+\infty.$
\item[(ii)] Assume that 
$(a_{n})$ and $(b_{n})$ are two sequences of positive reals tending to infinity,
$0<a_{1}<b_{1}<a_{2}<b_{2}<\cdot \cdot \cdot <a_{n} < b_{n}<\cdot \cdot \cdot,$ $c_{n}:=|{\mathbb N} \cap [a_{n},b_{n}]|$ ($n\in {\mathbb N}$) and $A:={\mathbb N} \cap \bigcup_{n\in {\mathbb N}}[a_{n},b_{n}].$ Let $l>a_{1}$ be given. Then there exist two cases: $a_{n}\leq l^{q}\leq b_{n}$ for some $n\in {\mathbb N},$ when
$|A \cap [1,l^{q}]|/l \leq a_{n}^{^{(-1)/q}}\sum_{j=1}^{n}c_{j}$
and $|A \cap [1,l^{q}]|/l \geq b_{n}^{^{(-1)/q}}\sum_{j=1}^{n}c_{j} \geq  a_{n+1}^{^{(-1)/q}}\sum_{j=1}^{n}c_{j}$
or $b_{n}< l^{q}< a_{n+1}$ for some $n\in {\mathbb N},$ when
$|A \cap [1,l^{q}]|/l \leq b_{n}^{^{(-1)/q}}\sum_{j=1}^{n}c_{j} \leq a_{n}^{^{(-1)/q}}\sum_{j=1}^{n}c_{j}  $
and $|A \cap [1,l^{q}]|/l \geq a_{n+1}^{^{(-1)/q}}\sum_{j=1}^{n}c_{j}.  $
Summing up, we get:
\begin{align}\label{qws}
\limsup_{n\rightarrow \infty}\frac{c_{1}+\cdot \cdot \cdot +c_{n}}{a_{n+1}^{1/q}} \leq \overline{d_{q}}(A)=\limsup_{l\rightarrow \infty}\frac{\bigl| A\cap [1,l^{q}] \bigr|}{l} \leq \limsup_{n\rightarrow \infty}\frac{c_{1}+\cdot \cdot \cdot +c_{n}}{a_{n}^{1/q}}.
\end{align}
Consider first the following case $a_{n}:=n^{2q}$ and $b_{n}:=n^{2q}+n$ ($n\in {\mathbb N},$ $n\geq n_{0},$ where $n_{0}\in {\mathbb N}$ is chosen so that $n^{2q}+n<(n+1)^{2q},$ $n\geq n_{0}$). Then \eqref{qws} yields that $\overline{d_{q}}(A)\leq 1.$ On the other hand, let $s_{0}>0$ be such that $(\lfloor s^{q} \rfloor -2)/s >1$ for $s\geq s_{0}.$ Then the set $A \cap [(a_{n}-1)+1, (a_{n}-1)+s^{q} ]$ has at least $\lfloor s^{q} \rfloor -2$ elements for $n\geq n_{1}$ suff. large, so that $\overline{Bd}_{l;q}(A)=\infty ,$ which cannot be expected in the case that $q=1$ (observe also that Proposition \ref{mile-label} implies that $\underline{Bd}_{u;q}(A)=0 $). 
Consider now the case in which $1<q'<q,$ $a_{n}:= \lfloor n^{q'} \rfloor$ and $b_{n}:=\lfloor n^{q'}\rfloor+1$ ($n\in {\mathbb N},$ $n\geq n_{2},$ where $n_{2}\in {\mathbb N}$ is chosen so that $\lfloor n^{q'} \rfloor +1< \lfloor(n+1)^{q'} \rfloor,$ $n\geq n_{2}$).
Then the first inequality in \eqref{qws} yields $\overline{d_{q}}(A)=\infty.$ On the other hand, for every $s>0$ and for every increasing sequence $(n_{l})$ tending to infinity, it is easily seen that there exists $l_{0}\in {\mathbb N}$ such that the interval $[n_{l}+1,n_{l}+s^{q}]$ contains at most one element of set $A$ for any $l\geq l_{0},$ which immediately implies that $\overline{Bd}_{l;q}(A)=0$ so that the inequality $\overline{d}_{q}(A)\leq \overline{Bd}_{l;q}(A)$ does not hold generally in the case that $q>1.$ 
\end{itemize}
\end{example}

We continue by stating the following result:

\begin{prop}\label{zapi}
Suppose that $A\subseteq {\mathbb N}$ and $(m_{n})$ is an increasing sequence  in $[1,\infty).$ Then the following holds:
\begin{itemize}
\item[(i)]
$\overline{Bd}_{u;{m_{n}}}(A)+\overline{Bd}_{u;{m_{n}}}({\mathbb N} \setminus A)\geq \limsup_{n \rightarrow \infty}\frac{m_{n}}{n} .$ Especially, if\\ $\limsup_{n \rightarrow \infty}\frac{m_{n}}{n}=\infty ,$ then 
$$
\overline{Bd}_{u;{m_{n}}}(A)+\overline{Bd}_{u;{m_{n}}}({\mathbb N} \setminus A)=\infty.
$$
\item[(ii)]
$\underline{Bd}_{l;{m_{n}}}(A)+\underline{Bd}_{l;{m_{n}}}({\mathbb N} \setminus A)\leq \liminf_{n \rightarrow \infty}\frac{m_{n}}{n} .$ Especially, if\\ $\liminf_{n \rightarrow \infty}\frac{m_{n}}{n}=0 ,$ then 
$$
\underline{Bd}_{l;{m_{n}}}(A)=\underline{Bd}_{l;{m_{n}}}({\mathbb N} \setminus A)=0.
$$
\item[(iii)] Suppose that $\lim_{n \rightarrow \infty}\frac{m_{n}}{n}$ does not exist in $[0,\infty].$ Then 
$$
\overline{Bd}_{u;{m_{n}}}(A)+\overline{Bd}_{u;{m_{n}}}({\mathbb N} \setminus A)>\underline{Bd}_{l;{m_{n}}}(A)+\underline{Bd}_{l;{m_{n}}}({\mathbb N} \setminus A)
$$
and particularly
$$
\overline{Bd}_{u;{m_{n}}}(A)>\underline{Bd}_{l;{m_{n}}}(A) \mbox{ or } \overline{Bd}_{u;{m_{n}}}({\mathbb N} \setminus A)>\underline{Bd}_{l;{m_{n}}}({\mathbb N} \setminus A).
$$
\end{itemize}
\end{prop}

\begin{proof}
It is clear that 
$$
\frac{\bigl|A \cap [n+1,n+m_{s}]\bigr|}{s}+\frac{\bigl| ({\mathbb N} \setminus A) \cap [n+1,n+m_{s}]\bigr|}{s}\geq \frac{m_{s}-1}{s},\quad s,\ n\in {\mathbb N}.
$$
Taking the limit superior as $n\rightarrow \infty,$ we get that, for every $s\in {\mathbb N},$
$$
\limsup_{n\rightarrow \infty}\frac{\bigl|A \cap [n+1,n+m_{s}]\bigr|}{s}+\limsup_{n\rightarrow \infty}\frac{\bigl| ({\mathbb N} \setminus A) \cap [n+1,n+m_{s}]\bigr|}{s}\geq \frac{m_{s}-1}{s}.
$$
Taking now the limit superior as $s\rightarrow \infty,$ we get
that 
\begin{align*}
& \limsup_{s\rightarrow \infty}\limsup_{n\rightarrow \infty}\frac{\bigl|A \cap [n+1,n+m_{s}]\bigr|}{s}+\limsup_{s\rightarrow \infty}\limsup_{n\rightarrow \infty}\frac{\bigl| ({\mathbb N} \setminus A) \cap [n+1,n+m_{s}]\bigr|}{s}
\\ & \geq \limsup_{s\rightarrow \infty}\frac{m_{s}-1}{s}=\limsup_{s\rightarrow \infty}\frac{m_{s}}{s},
\end{align*}
which clearly implies the part (i). The assertion (ii) can be deduced similarly, by observing that
$$
\frac{\bigl|A \cap [n+1,n+m_{s}]\bigr|}{s}+\frac{\bigl| ({\mathbb N} \setminus A) \cap [n+1,n+m_{s}]\bigr|}{s}\leq \frac{m_{s}+1}{s},\quad s,\ n\in {\mathbb N}
$$
and taking after that the limit inferiors as $n\rightarrow \infty$ and $s\rightarrow \infty .$ Suppose now that $\lim_{n \rightarrow \infty}\frac{m_{n}}{n}$ does not exist. Then $\limsup_{n \rightarrow \infty}\frac{m_{n}}{n}>\liminf_{n \rightarrow \infty}\frac{m_{n}}{n}$ and the part (iii) follows by applying (i)-(ii).
\end{proof}

\begin{prop}\label{zapiw}
Suppose that $A,\ B\subseteq {\mathbb N}$ and $(m_{n})$ is an increasing sequence  in $[1,\infty).$ Then 
$$
\overline{Bd}_{u;{m_{n}}}(A \cup B)\leq \overline{Bd}_{u;{m_{n}}}(A)+\overline{Bd}_{u;{m_{n}}}(B).
$$
\end{prop}

\begin{proof}
Since for each $n\in {\mathbb N}$ we have $| (A\cup B) \cap  [n+1,n+m_{s}]|=| (A\cap  [n+1,n+m_{s}]) \cup (B\cap  [n+1,n+m_{s}])| \leq |A\cap  [n+1,n+m_{s}]| +|B\cap  [n+1,n+m_{s}]| ,$
the required inequality follows by taking first the limit superior as $n\rightarrow \infty$ and after that the limit superior as $s\rightarrow \infty.$
\end{proof}

In a similar fashion, we can deduce a great number of other inequalities for the introduced Banach densities. For example, if $A,\ B\subseteq {\mathbb N}$ and $(m_{n})$ is an increasing sequence  in $[1,\infty),$ then we have  
$$
\underline{Bd}_{l;{m_{n}}}(A \Delta B)+ \underline{Bd}_{l;{m_{n}}}(A \cap B)\leq \underline{Bd}_{l;{m_{n}}}(A\cup B),
$$
where $A \Delta B$ denotes the symmetric difference of sets $A$ and $B.$

\begin{prop}\label{profica}
Suppose that $\lim_{n \rightarrow \infty}\frac{m_{n}}{n}=k\in [0,\infty).$ Then we have $
\overline{Bd}_{l;{m_{n}}}(A)=
\overline{Bd}_{u;{m_{n}}}(A)=
k\overline{Bd}(A)$ and $
\underline{Bd}_{l;{m_{n}}}(A)=
\underline{Bd}_{u;{m_{n}}}(A)=
k\underline{Bd}(A).$
\end{prop}

\begin{proof}
We will prove the statement of proposition only in the case that $k>0.$
For any $c_{-}\in (0,k)$ and $c_{+}\in (k,\infty),$ we have the existence of a positive integer $l\in {\mathbb N}$ such that $c_{-}n\leq m_{n}\leq c_{+}n,$ $n\geq l.$ Hence,
$$
\frac{\bigl| A \cap [n+1,n+m_{s}]\bigr|}{s}\leq c_{+}\frac{\bigl|  A \cap [n+1,n+c_{+}s]\bigr|}{c_{+}s},\quad s \geq l
$$ 
and
$$
\frac{\bigl|  A \cap [n+1,n+m_{s}]\bigr|}{s}\geq c_{-}\frac{\bigl| A \cap [n+1,n+c_{-}s]\bigr|}{c_{-}s},\quad s\geq l.
$$ 
Taking into account \cite[Theorem 1.3]{grekos} and elementary substitutions $s\mapsto c_{\pm} s,$ the above implies
\begin{align*}
\limsup_{s\rightarrow \infty}&  \limsup_{n\rightarrow \infty}\frac{\bigl| A \cap [n+1,n+m_{s}]\bigr|}{s}\leq c_{+}\limsup_{s\rightarrow \infty} \limsup_{n\rightarrow \infty}\frac{\bigl|  A \cap [n+1,n+c_{+}s]\bigr|}{c_{+}s}
\\ & =c_{+}\limsup_{s\rightarrow \infty} \limsup_{n\rightarrow \infty}\frac{\bigl|  A \cap [n+1,n+s]\bigr|}{s}=c_{+}\overline{Bd}(A)
\end{align*}
and
\begin{align*}
\liminf_{s\rightarrow \infty}& \limsup_{n\rightarrow \infty}\frac{\bigl|  A \cap [n+1,n+m_{s}]\bigr|}{s}\geq c_{-}\liminf_{s\rightarrow \infty} \limsup_{n\rightarrow \infty}\frac{\bigl|  A \cap [n+1,n+c_{-}s]\bigr|}{c_{-}s}
\\ & =c_{-}\liminf_{s\rightarrow \infty} \limsup_{n\rightarrow \infty}\frac{\bigl|  A \cap [n+1,n+s]\bigr|}{s}=c_{-}\overline{Bd}(A),
\end{align*}
so that the equalities $
\overline{Bd}_{l;{m_{n}}}(A)=
\overline{Bd}_{u;{m_{n}}}(A)=
k\overline{Bd}(A)$ follows by letting $c_{\pm} \rightarrow k.$ The equalities $
\underline{Bd}_{l;{m_{n}}}(A)=
\underline{Bd}_{u;{m_{n}}}(A)=
k\underline{Bd}(A)$ can be proved similarly, by appealing to \cite[Proposition 3.1]{grekos} in place of \cite[Theorem 1.3]{grekos}. 
\end{proof}

Combining Proposition \ref{mile} and Proposition \ref{profica}, we get that the existence of $\lim_{n \rightarrow \infty}\frac{m_{n}}{n}$ in $[0,\infty]$ implies $
\underline{Bd}_{l;{m_{n}}}(A)=
\underline{Bd}_{u;{m_{n}}}(A)$ for any $A\subseteq {\mathbb N}.$ In the case that $\lim_{n \rightarrow \infty}\frac{m_{n}}{n}$ does not exist in $[0,\infty],$ we have $\liminf_{n \rightarrow \infty}\frac{m_{n}}{n} <\limsup_{n \rightarrow \infty}\frac{m_{n}}{n}$ and therefore
$$
\underline{d}_{{m_{n}}}({\mathbb N})=\liminf_{s \rightarrow \infty}\frac{m_{s}}{s} =\underline{Bd}_{l;{m_{n}}}({\mathbb N})<\underline{Bd}_{u;{m_{n}}}({\mathbb N})=\limsup_{s\rightarrow \infty}\frac{m_{s}}{s}=\overline{d}_{{m_{n}}}({\mathbb N}).
$$

By the foregoing, we have the following corollary:

\begin{cor}\label{zar}
Suppose that $(m_{n})$ is an increasing sequence  in $[1,\infty).$ Then the following statements are equivalent:
\begin{itemize}
\item[(i)] 
$\underline{Bd}_{l;{m_{n}}}(A)=
\underline{Bd}_{u;{m_{n}}}(A)$ for all subsets $A$ of ${\mathbb N}.$
\item[(ii)] $\underline{Bd}_{l;{m_{n}}}({\mathbb N})=
\underline{Bd}_{u;{m_{n}}}({\mathbb N}).$
\item[(iii)] The limit $\lim_{n \rightarrow \infty}\frac{m_{n}}{n}$ exists in $[0,\infty].$
\end{itemize}
\end{cor}

\section{${\mathcal F}$-hypercyclicity of linear operators on Fr\' echet spaces}\label{yunied}

Assume that $q\in [1,\infty)$ and $(m_{n})$ is an increasing sequence  in $[1,\infty).$ Consider the notion introduced in Definition \ref{4-skins-MLO-okay} with:
\begin{itemize}
\item[(i)] ${\mathcal F}=\{A \subseteq {\mathbb N} : \underline{d}(A)>0\},$     (ii) ${\mathcal F}=\{A \subseteq {\mathbb N} : \underline{d}_{q}(A)>0\},$
\item[(iii)] ${\mathcal F}=\{A \subseteq {\mathbb N} : \overline{d}(A)>0\},$     (iv) ${\mathcal F}=\{A \subseteq {\mathbb N} : \overline{d}_{q}(A)>0\},$
\item[(v)] ${\mathcal F}=\{A \subseteq {\mathbb N} : \underline{d}_{{m_{n}}}(A)>0\},$ (vi) ${\mathcal F}=\{A \subseteq {\mathbb N} : \overline{d}_{{m_{n}}}(A)>0\},$
\item[(vii)] ${\mathcal F}=\{A \subseteq {\mathbb N} : \underline{Bd}(A)>0\},$ (viii) ${\mathcal F}=\{A \subseteq {\mathbb N} : \overline{Bd}(A)>0\};$
\end{itemize}
then we say that $(T_{n})_{n\in {\mathbb N}}$ ($T,$ $x$) is frequently hypercyclic, $q$-frequently hypercyclic, upper frequently hypercyclic, upper $q$-frequently hypercyclic,
l-$(m_{n})$-hypercyclic, u-$(m_{n})$-hypercyclic, lower reiteratively hypercyclic and reiteratively hypercyclic, respectively. 

Due to \eqref{nejes}, we have that the frequent hypercyclicity implies $q$-hypercyclicity for any $q\geq 1.$ It is very simple to prove a great number of other, almost trivial, inequalities for lower and upper densities. For example, let $(m_{n}')$ be an increasing sequence of positive integers and $m_{n}\leq m_{n}'$ for all $n\in {\mathbb N}.$ Then 
$$
\underline{Bd}_{l;q}(A) \leq \underline{Bd}_{u;q}(A) \mbox{ and } \overline{Bd}_{l;q}(A) \leq \overline{Bd}_{u;q}(A),\quad q\geq 1;
$$
$$
\underline{Bd}_{l;m_{n}}(A) \leq \underline{Bd}_{u;m_{n}'}(A) \mbox{ and } \overline{Bd}_{l;m_{n}}(A) \leq \overline{Bd}_{u;m_{n}'}(A),\quad q\geq 1;
$$
$$
\underline{Bd}_{l;q_{1}}(A) \leq \underline{Bd}_{l;q_{2}}(A) \mbox{ and } \underline{Bd}_{u;q_{1}}(A) \leq \underline{Bd}_{u;q_{2}}(A),\quad 1\leq q_{1}\leq q_{2}<\infty;
$$
$$
\overline{Bd}_{l;q_{1}}(A) \leq \overline{Bd}_{l;q_{2}}(A) \mbox{ and } \overline{Bd}_{u;q_{1}}(A) \leq \overline{Bd}_{u;q_{2}}(A),\quad 1\leq q_{1}\leq q_{2}<\infty;
$$
$$
\underline{Bd}_{l;m_{n}}(A) \leq \underline{Bd}_{l;m_{n}'}(A) \mbox{ and } \underline{Bd}_{u;m_{n}}(A) \leq \underline{Bd}_{u;m_{n}'}(A);
$$
$$
\overline{Bd}_{l;m_{n}}(A) \leq \overline{Bd}_{l;m_{n}'}(A) \mbox{ and } \overline{Bd}_{u;m_{n}}(A) \leq \overline{Bd}_{u;m_{n}'}(A),
$$
which implies a great number of obvious consequences. For example, the lower $l;q_{1}$-reiterative hypercyclicity implies the lower $l;q_{2}$-reiterative hypercyclicity for $1\leq q_{1}\leq q_{2}<\infty .$

Now we would like to correct 
a small inconsistency made in \cite{kimpark1}, where the author has stated that the lower $(m_{k})$-density of set $A=\{ n_{1},\ n_{2}, \cdot \cdot \cdot, n_{k},\cdot \cdot \cdot \},$ where $(n_{k})$ is a strictly increasing sequence of positive integers, satisfies
$\underline{d}_{m_{k}}(A)>0$ iff there exists a finite constant $L_{1}>0$ 
such that $n_{k}\leq L_{1}m_{k},$ $k\in {\mathbb N}.$ This is not true in general and a simple counterexample is given by $n_{k}:=ke^{k}$ and $m_{k}:=e^{k}$ ($k\in {\mathbb N}$): then 
the lower $(m_{k})$-density of set $A$ is equal to $1$ but there is no finite constant  $L_{1}>0$ 
such that $n_{k}\leq L_{1}m_{k},$ $k\in {\mathbb N}.$ Therefore, we need to make a strict difference between the notion of l-$(m_{k})$-hypercyclicity and the notion of $(m_{k})$-hypercyclicity. Concerning this question, we have the following remark:

\begin{rem}\label{hashima}
It is very straightforward to see that the $(m_{k})$-hypercyclicity automatically implies l-$(m_{k})$-hypercyclicity, provided the validity of following condition:
$$
\mbox{For all } L>0,\mbox{ we have } \liminf_{k\rightarrow \infty}\frac{\sup\{ s\in {\mathbb N} : Lm_{s}\leq m_{k} \}}{k}>0.
$$
On the other hand, the l-$(m_{k})$-hypercyclicity implies $(m_{k})$-hypercyclicity, provided the validity of following condition:
\begin{align}\label{otkrio}
\mbox{For all }c>0,\mbox{ there exists }L>0\mbox{ such that } m_{\lceil kc \rceil} \leq Lm_{k}\mbox{ for all }k\in {\mathbb N};
\end{align}
to see this, it suffices to observe that 
the l-$(m_{k})$-hypercyclicity of a linear operator $T$ on $X,$ say, implies that for every open non-empty subset $U$ of $X$ there exists a positive constant $c>0$ such that
the interval $[1,m_{\lceil kc \rceil}]$ contains at least $k$ elements of set $A:=\{s \in {\mathbb N} : T^{s}x\in U\},$ where $x$ is an l-$(m_{k})$-hypercyclic vector for $T.$ Then \eqref{otkrio} yields that we can find a strictly increasing sequence $(n_{k})$ of positive integers, where $n_{k}$ is the $k$-th element of $A$ in increasing order, such that $n_{k}\leq m_{\lceil kc \rceil}\leq Lm_{k},$ $k\in {\mathbb N}.$
\end{rem}

For more details about examples of hypercyclic operators that are not frequently hypercyclic, frequently hypercyclic operators that are not upper frequently hypercyclic and similar problematic, the reader may consult \cite{puid-dios-erg} and references cited therein.

If $q\in [1,\infty),$ $(m_{n})$ is an increasing sequence  in $[1,\infty)$, and
\begin{itemize}
\item[(i)'] ${\mathcal F}=\{A \subseteq {\mathbb N} : \underline{Bd}_{l;q}(A)>0\},$ (ii)' ${\mathcal F}=\{A \subseteq {\mathbb N} : \underline{Bd}_{u;q}(A)>0\},$
\item[(iii)'] ${\mathcal F}=\{A \subseteq {\mathbb N} : \overline{Bd}_{l;q}(A)\},$ (iv)' ${\mathcal F}=\{A \subseteq {\mathbb N} : \overline{Bd}_{u;q}(A)\},$
\item[(v)'] ${\mathcal F}=\{A \subseteq {\mathbb N} : \underline{Bd}_{l;{m_{n}}}(A)>0\},$ (vi)' ${\mathcal F}=\{A \subseteq {\mathbb N} : \underline{Bd}_{u;{m_{n}}}(A)>0\},$
\item[(vii)'] ${\mathcal F}=\{A \subseteq {\mathbb N} : \overline{Bd}_{l;{m_{n}}}(A)>0\},$ (viii)' ${\mathcal F}=\{A \subseteq {\mathbb N} : \overline{Bd}_{u;{m_{n}}}(A)>0\},$
\end{itemize}
then we say that $(T_{n})_{n\in {\mathbb N}}$ ($T,$ $x$) is lower $l;q$-reiteratively hypercyclic, lower 
$u;q$-reiteratively hypercyclic, upper $l;q$-reiteratively hypercyclic, upper $u;q$-reiteratively hypercyclic, lower $l;m_{n}$-reiteratively hypercyclic, lower $u;m_{n}$-reiteratively hypercyclic, upper $l;m_{n}$-reiteratively hypercyclic and upper $u;m_{n}$-reiteratively hypercyclic, respectively
(see Definition \ref{4-skins-MLO-okay}). 

\begin{rem}\label{novi-remark}
Although we do not intend to analyze ${\mathcal F}$-transitive linear operators in separable Fr\' echet spaces, we would like to note that,
in any case (i)'-(viii)' set out above, the family $\tilde{{\mathcal F}}$ is not a  filter so that ${\mathcal F}$-Transitivity Criterion \cite[Theorem 2.4]{biba} is inapplicable (see \cite{biba} for the notion).
The families defined in (v) and (vi) are (proper) Furstenberg families iff $\liminf_{s\rightarrow \infty}\frac{m_{s}}{s}>0,$ while the families defined in (vii) and (viii) are (proper) Furstenberg families iff $\limsup_{s\rightarrow \infty}\frac{m_{s}}{s}>0.$ It is not clear whether the condition $\liminf_{s\rightarrow \infty}\frac{m_{s}}{s}>0$ ensures that the families defined in (iii)'-(iv)' and (vii)'-(viii)' are
upper Furstenberg; we will only note that any $\liminf_{s\rightarrow \infty}$, $\liminf_{n\rightarrow \infty},$ $\limsup_{s\rightarrow \infty}$ or $\limsup_{n\rightarrow \infty}$ in Definition \ref{guptinjo} can be replaced, optionally, with $\inf_{s\in {\mathbb N}},$  $\inf_{n\in {\mathbb N}},$  $\sup_{s\in {\mathbb N}}$ or $\sup_{n\in {\mathbb N}},$
respectively, which should be much better choice if we want to apply results from \cite{boni-upper}; see, especially, \cite[Example 12(b)]{boni-upper}. We will consider this question in more details somewhere else.
\end{rem}

\begin{rem}\label{novi-remarkk}
Due to Theorem \ref{mile-prim} and Corollary \ref{mile-duo}(ii), if $\lim_{n\rightarrow +\infty}\frac{m_{n}}{n}=+\infty$ and $q>1,$ then the concepts of l-$(m_{n})$-hypercyclicity, lower $l;m_{n}$-reiterative hypercyclicity and lower $u;m_{n}$-reiterative hypercyclicity
coincide, which continues to hold for the concepts of $q$-frequent hypercyclicity, lower $l;q$-reiterative hypercyclicity and lower 
$u;q$-reiterative hypercyclicity (let us recall that Bayart and Matheron have proved \cite[Theorem 1.2, Corollary 1.3]{Bay} that the condition $\lim_{n\rightarrow +\infty}\frac{m_{n}}{n}=+\infty$ enables one to construct a linear continuous operator on $l^{1}({\mathbb N})$ that is $(m_{n})$-hypercyclic but not weakly mixing, as well as that for any $q>1$ there exists a linear continuous operator on $l^{1}({\mathbb N})$ that is $q$-frequently hypercyclic but not weakly mixing). In the case that $q=1,$ then there is no separable Fr\' echet space supporting a lower $l;1$-reiteratively hypercyclic (lower $u;1$-reiteratively hypercyclic) linear continuous operator due to a recent result of B\'es, Menet, Peris  and Puig \cite{progres-regres}. This statement is easily transferred to linear not necessarily continuous operators in Fr\'echet spaces (see also Theorem \ref{debos} below).
\end{rem}

We continue by observing that, in spite of Proposition \ref{perhane}, an element $x\in X$ is a $q$-frequently hypercyclic vector for a continuous linear operator $T\in L(X)$ iff for every non-empty open set $U$ of $X$ there exist a strictly increasing sequence $(n_{k})$ of natural numbers and a finite constant $L>0$ such that $n_{k}\leq Lk^{q},$ $k\in {\mathbb N}$ and $T^{n_{k}}x\in U,$ $k\in {\mathbb N}$ ($q\geq 1$).
Here, it is worth observing that the $q$-Frequent Hypercyclic Criterion \cite[Theorem 2.5]{heo} continues to hold for any number 
$q\geq 1$ if we replace the terms $k ^{q},$ $ (k-n) ^{q}$ and $(k+n) ^{q}$ in the formulation of this theorem with terms $\lfloor k ^{q} \rfloor,$ $\lfloor (k-n) ^{q} \rfloor$ and $\lfloor (k+n) ^{q} \rfloor,$ respectively ($\lceil k ^{q} \rceil ,$ $\lceil (k-n) ^{q} \rceil$ and $\lceil (k+n) ^{q} \rceil ,$ respectively),
which can be proved with the help of Lemma \ref{zelje-zelje} and the corresponding result known in the case that $q\in {\mathbb N};$ similarly, we can reformulate the statement of
\cite[Theorem 3.2]{kimpark1} for any increasing sequence $(m_{k})$ in $[1,\infty).$ 
Speaking-matter-of-factly, the following extension of \cite[Theorem 3.2]{kimpark1}, holding for sequences of operators, can be proved by applying directly Frequent Universality Criterion \cite[Theorem 2.4]{boni-ERAT}. The theorem below also provides a slight generalization of Frequent Universality Criterion:

\begin{thm}\label{mk-FUC} (l-$(m_{k})$-Frequent Universality Criterion)
Suppose that $X$ is Fr\' echet space, $Y$ is a separable Fr\' echet space, $T_{n}\in L(X,Y)$ and $(m_{k})$ is an increasing sequence in $[1,\infty).$ 
Suppose that there is a dense subset $Y_{0}$ of $Y$ and mappings $S_{n} : Y_{0} \rightarrow X$ ($n\in {\mathbb N}$) such that
the following conditions hold for all $y\in Y_{0}$:
\begin{itemize} 
\item[(i)]The series $\sum_{n=1}^{k}T_{\lfloor m_{k}\rfloor}S_{\lfloor m_{k-n}\rfloor}y$ converges unconditionally, uniformly in $k\in {\mathbb N}.$
\item[(ii)] The series $\sum_{n=1}^{\infty}T_{\lfloor m_{k}\rfloor}S_{\lfloor m_{k+n} \rfloor}y$ converges unconditionally, uniformly in $k\in {\mathbb N}.$
\item[(iii)] The series $\sum_{n=1}^{\infty}S_{\lfloor m_{n} \rfloor}y$ converges unconditionally, uniformly in $n\in {\mathbb N}.$
\item[(iv)] $\lim_{n\rightarrow \infty}T_{\lfloor m_{n} \rfloor}S_{\lfloor m_{n} \rfloor}y=y,$ $y\in Y_{0}.$
\end{itemize}
Then the sequence of operators $(T_{n})_{n\in {\mathbb N}}$ is l-$(m_{k})$-frequently hypercyclic.
\end{thm}

\begin{proof}
Without loss of generality, we may assume that $m_{k}\in {\mathbb N},$ $k\in {\mathbb N}.$
Set ${\mathcal T}_{n}:=T_{m_{n}}$ and ${\mathcal S}_{n}:=S_{m_{n}}$ ($n\in {\mathbb N}$). Then the assumptions (i)-(iv) yield that we can apply Frequent Universality Criterion to ${\mathcal T}_{n}$ and ${\mathcal S}_{n}$ ($n\in {\mathbb N}$), so that ${\mathcal T}_{n}$ 
is frequently universal, i.e., there exists an element $x\in X$ satisfying that, for every open non-empty subset $U$ of $X,$ we can find a strictly increasing sequence $(n_{k})$ of positive integers and a finite constant $L\geq 1$ such that $n_{k}\leq Lk,$ $k\in {\mathbb N}$ and $T^{m_{n_{k}}}x\in U,$ $k\in {\mathbb N}.$
Since $m_{n_{k_{0}}}\leq m_{Lk_{0}}\leq m_{k}$ for $1\leq k_{0}\leq \lfloor k/L\rfloor,$ we have that for this set $U$ and each $k\in {\mathbb N}$ we have $\{m_{n_{i}} : 1\leq i\leq \lfloor k/L\rfloor\} \subseteq [1,m_{k}]$ and $T_{m_{n_{i}}}x\in U$ so that the lower $(m_{k})$-density of set $\{ n\in {\mathbb N} : T_{n}x \in U \}$ is greater than or equal to $\lfloor 1/L \rfloor .$
This completes the proof.
\end{proof}

In a series of our researches carried out so far, we have investigated dynamical properties of unbounded linear operators following the ideas and methods initiated in \cite{cycch} (see \cite{knjigah}-\cite{knjigaho} for similar results). The subsequent result
can be deduced by applying l-$(m_{k})$-Frequent Universality Criterion to the sequence $T_{n}\in L([R(C)]_{\Theta},X)$ defined by $T_{n}Cx:=T^{n}Cx,$ $n\in {\mathbb N},$ $x\in X:$ 

\begin{thm}\label{mk-FUC} (l-$(m_{k})$-Frequent Universality Criterion for Unbounded Linear Operators)
Suppose that $X$ is a separable Fr\' echet space, $T$ is a linear operator acting on $X$, and $(m_{k})$ is an increasing sequence in $[1,\infty).$ Let $C\in L(X)$ be injective.
Suppose that there is a dense subset $X_{0}$ of $X$ and mappings $S_{n} : X_{0} \rightarrow R(C)$ ($n\in {\mathbb N}$) such that
the following conditions hold for all $x\in X_{0}$:
\begin{itemize} 
\item[(i)]The series $\sum_{n=1}^{k}T^{\lfloor m_{k}\rfloor} S_{\lfloor m_{k-n}\rfloor}x$ converges unconditionally, uniformly in $k\in {\mathbb N}.$
\item[(ii)] The series $\sum_{n=1}^{\infty}T^{\lfloor m_{k}\rfloor} S_{\lfloor m_{k+n} \rfloor}x$ converges unconditionally, uniformly in $k\in {\mathbb N}.$
\item[(iii)] The series $\sum_{n=1}^{\infty}S_{\lfloor m_{n} \rfloor}x$ converges unconditionally, uniformly in $n\in {\mathbb N}.$
\item[(iv)] $\lim_{n\rightarrow \infty}T^{\lfloor m_{n} \rfloor} S_{\lfloor m_{n} \rfloor}x=x,$ $x\in X_{0}.$
\item[(v)] $R(C) \subseteq D_{\infty}(T)$ and $T^{n}C\in L(X),$ $n\in {\mathbb N}.$ 
\end{itemize}
Then the operator $T$ is l-$(m_{k})$-frequently hypercyclic.
\end{thm}

In the following theorem, we state
a slight extension of Upper Frequent Hypercyclicity Criterion \cite[Corollary 24]{boni-upper} for unbounded linear operators (a straightforward proof is omitted):

\begin{thm}\label{C-unbounded} (Upper Frequent Hypercyclicity Criterion for Unbounded Linear Operators)
Suppose that $X$ is a separable Fr\' echet space, $T$ is a linear operator acting on $X,$ ${\mathcal F}$ is an upper Furstenberg family and $C\in L(X)$ is injective.
Suppose that there are two dense subsets $X_{0}$ and $Y_{0}$ of $X$ and mappings $S_{n} : Y_{0} \rightarrow R(C)$ ($n\in {\mathbb N}$) such that
the following condition holds: For every $y\in Y_{0}$ and $\epsilon>0,$ there exist $A\in {\mathcal F}$ and $\delta \in D$ such that:
\begin{itemize} 
\item[(i)] For every $x\in X_{0},$ there exists $B\in {\mathcal F}_{\delta}$ such that $B\subseteq A$ and that, for every $n\in {\mathbb N},$ $\|T^{n}Cx\| <\epsilon .$  
\item[(ii)] The series $\sum_{n\in A}S_{n}y$ converges.
\item[(iii)] For every $m\in A,$ we have $\|T^{m}\sum_{n\in A}S_{n}y -y\|<\epsilon.$
\item[(iv)] $R(C) \subseteq D_{\infty}(T)$ and $T^{n}C\in L(X),$ $n\in {\mathbb N}.$ 
\end{itemize}
Then the operator $T$ is ${\mathcal F}$-hypercyclic and the set consisting of all ${\mathcal F}$-hypercyclic vectors of operator $T$ is residual in $[R(C)]_{\circledast}.$
\end{thm}

Before proceeding further, we would like to note that the assertions of \cite[Theorem 5.1, Corollary 5.2]{boni} continue to hold for sequences of continuous linear operators acting between different pivot spaces as well as that the l-$(m_{k})$-Frequent Universality Criterion for Unbounded Linear Operators
can be very difficultly applied directly. In the case of usually considered frequent hypercyclicity, this criterion has been essentially applied in the proof of \cite[Theorem 9.22]{Grosse}. The main aim of following theorem is to reconsider the last mentioned result for densely defined linear operators with non-empty resolvent set by using an idea of Kalmes \cite{dokserv}:

\begin{thm}\label{debos}
Suppose that $X$ is a separable Fr\' echet space, $T$ is a densely defined linear operator acting on $X$ and $\rho(T)\neq \emptyset.$ 
Let $J$ be a non-empty family, and let $E_{j} : {\mathbb T}\equiv \{ z\in {\mathbb C} : |z|=1\} \rightarrow X$ be a continuous mapping, resp. $C^{2}$-mapping, with $TE_{j}(z)=zE_{j}(z),$ $z\in {\mathbb T}$ ($j\in J$), satisfying that
$$
X_{0}\equiv span\{E_{j}(z) :  z\in {\mathbb T},\ j\in J\}
$$  
is dense in $X$. If there exists a sequence of continuous linear operators $(D_{s})$ in $ L(X)$ such that $TD_{s}\subseteq D_{s}T,$ $T^{n}D_{s}\in L(X)$ for all $s,\ n\in {\mathbb N}$ 
and $\bigcup_{s\in {\mathbb N}}R(D_{s})$ is dense in $[D_{\infty}(T)],$ then the operator $T$ is topologically mixing and chaotic, resp. frequently hypercyclic.
\end{thm}

\begin{proof}
Since $\rho(T)\neq \emptyset ,$ we have that the operator $T^{n}$ is closed and the space $[D(T^{n})]$ is a Fr\' echet one ($n\in {\mathbb N}$). Applying the abstract Mittag-Leffler theorem \cite[Theorem 6.3]{local}, we can deduce that $D_{\infty}(T)$ is dense in the initial space $X$ and any space $[D(T^{n})]$ ($n\in {\mathbb N}$); see also the proof of  \cite[Proposition 6.2]{local}.
Furthermore, if $\{x_{k} : k\in {\mathbb N} \}$ is a dense sequence in $X,$ then it can be easily seen with the help of our assumption $T^{n}D_{s}\in L(E)$ ($s,\ n\in {\mathbb N}$) and the density of 
$\bigcup_{s\in {\mathbb N}}R(D_{s})$ in $[D_{\infty}(T)]$ that $\{D_{s}x_{k} : s,\ k\in {\mathbb N}\}$ is a dense sequence in $[D_{\infty}(T)],$ so that $[D_{\infty}(T)]$ is separable. It is clear that the restriction of operator $T$ to $D_{\infty}(T),$ $T_{\infty}$ for short, is a continuous linear operator on $D_{\infty}(T),$
$T_{\infty}\in L([D_{\infty}(T)]).$ We have that
$T_{\infty}D_{s}E_{j}(z)=D_{s}T_{\infty}E_{j}(z)=zD_{s}E_{j}(z),$ $z\in {\mathbb T}$ ($s\in {\mathbb N}$, $j\in J$), as well as that $D_{s}E_{j} : {\mathbb T}\equiv \{ z\in {\mathbb C} : |z|=1\} \rightarrow D_{\infty}(T)$ is a continuous mapping, resp. $C^{2}$-mapping ($s\in {\mathbb N},$ $j\in J$) and
$$
X_{0}^{\infty}\equiv span\bigl\{D_{s}E_{j}(z) :  z\in {\mathbb T},\ s\in {\mathbb N},\ j\in J\bigr\}
$$  
is dense in $[D_{\infty}(T)].$ An application of \cite[Theorem 9.22]{Grosse} yields that the operator $T_{\infty}$ is topologically mixing and chaotic, resp. frequently hypercyclic, which clearly implies that the operator $T$ has the same properties since the inclusion mapping $\tau : D_{\infty}(T) \rightarrow X$ has dense range and $T \circ \tau=\tau \circ T_{\infty}.$
\end{proof}

\begin{rem}\label{bvcp}
Using the argumentation contained in the proof \cite[Theorem 3.13]{266}, it can be easily seen that Theorem \ref{debos} provides a proper extension of \cite[Theorem 9.22]{Grosse}
in Banach spaces.
\end{rem}

In \cite[Example 3.8-Example 3.10]{266}, we have investigated densely distributionally chaotic properties of unbounded differential operators in Banach spaces and the main purpose of following theorem is to show how we can prove the frequent hypercyclicity and chaoticity of these operators by utilizing Theorem \ref{debos}:

\begin{example}\label{kraj}
Let $T$ be any of the operators analyzed in \cite[Example 3.8-Example 3.10]{266}. Then $T$ is densely defined, $\rho(T)\neq \emptyset$ 
and there exist numbers $0< a <b <2\pi$ and infinitely differentiable mappings $F_{1} : \{ e^{it} : t\in (a,b)\} \rightarrow X$ and $F_{2}: \{ e^{it} : t\in (a,b)\} \rightarrow X$
such that $F_{1}(e^{it})\in N(e^{it}-T)$ and $F_{2}(e^{it})\in N(e^{it}-T)$ for all $t\in [a,b]
,$ as well as that, for $a<a'<b'<b,$ the linear span of set $\{F_{1}(z) : z=e^{it}\mbox{ for some }t\in (a',b')  \} \cup \{ F_{2}(z) : z=e^{it}\mbox{ for some }t\in (a',b') \} $ is dense in $E.$
Set $E_{i}(e^{it}):=(t-a)^{3}(t-b)^{3}F_{i}(e^{it})$ for $t\in [a,b]$ and $E_{i}(e^{it}):=0$ for $t\in [0,a] \cap [b,2\pi],$ $i=1,2.$ Then $E_{i}(\cdot)$ is twice continuously differentiable and the subspace
$X_{0}
$
is dense in $X$. Furthermore, the conditions (i)-(iii) from the formulation of \cite[Theorem 3.13]{266} hold with $C=I,$ the identity operator on $X,$ and the requirements of Theorem \ref{debos} are satisfied with the sequence of operators $D_{s}:=C_{b}(1/s),$ where $s\geq 2$ and the operator $C_{b}(1/s)$ has the same meaning as in the proof of \cite[Theorem 3.13]{266}.
Hence, the operator $T$ is frequently hypercyclic, chaotic and topologically mixing. 

For illustration,
suppose that $X:=L^{2}({\mathbb R}),$ $c>b/2>0,$ $\Omega:=\{ \lambda \in
{\mathbb C} : \Re \lambda<c-b/2\}$ and ${\mathcal T}_{c}u:=u^{\prime
\prime}+2bxu^{\prime}+cu$ is the bounded perturbation of the
one-dimensional Ornstein-Uhlenbeck operator\index{operator!Ornstein-Uhlenbeck} acting with domain
$D({\mathcal T}_{c}):=\{u\in L^{2}({\mathbb R})\ \cap \
W^{2,2}_{loc}({\mathbb R}) : {\mathcal T}_{c}u\in L^{2}({\mathbb
R})\};$ see \cite{conman} and \cite{metafune} for more details.
Then the operator ${\mathcal T}_{c}$ is frequently hypercyclic, chaotic and topologically mixing, which continue to hold for the
multi-dimensional Ornstein-Uhlenbeck operators from
\cite[Section 4]{conman}. If we suppose that $r>0,$ $\sigma>0,$ $\nu =\sigma/\sqrt{2},$ $\gamma=r/\mu -\mu,$ $s>1,$ $s\nu>1$ and $\tau \geq 0,$ the vector space
$$
Y^{s,\tau}:=\Biggl\{ u\in C((0,\infty)) : \lim \limits_{x\rightarrow 0}\frac{u(x)}{1+x^{-\tau}}=\lim \limits_{x\rightarrow \infty}\frac{u(x)}{1+x^{s}}=0  \Biggr\},
$$
equipped with the norm
$$
\|u\|_{s,\tau}:=\sup_{x>0}\Biggl|\frac{u(x)}{\bigl(1+x^{-\tau}\bigr)\bigl(1+x^{s}\bigr)}\Biggr|,\quad u\in Y^{s,\tau},
$$
is a separable Banach space. We can similarly prove that the Black-Scholes operator $T,$ defined by $T:=D_{\nu}^{2}+\gamma D_{\mu}-r$ and acting with its maximal domain in $Y^{s,\tau}$, where 
$D_{\mu}:=\nu xd/dx,$ is frequently hypercyclic, chaotic and topologically mixing (\cite{emamirad-goldstein}).
\end{example}

Further analyses of ${\mathcal F}$-hypercyclic operators for the families ${\mathcal F}$ taking the form (i)'-(viii)' fall out from the scope of this paper.

\end{document}